\documentclass[reqno,11pt]{amsart}

\usepackage{amsmath,amssymb,amsthm,mathrsfs}
\usepackage{epsfig,color}
\usepackage[latin1]{inputenc}
\usepackage[english]{babel}
\usepackage[numbers]{natbib}

\numberwithin{equation}{section}

\def\R{\mathbb R}

\def\Om{\Omega}

\def\pa{\partial}

\def\00{{\bf 0}}

\newtheorem*{theorem*}{Theorem}
\newtheorem{theorem}{Theorem}[section]

\newtheorem{proposition}[theorem]{Proposition}

\newtheorem{corollary}[theorem]{Corollary}
\newtheorem{remark}[theorem]{Remark}

\title[]{A rigidity problem on the round sphere}

\author{Giulio Ciraolo and Luigi Vezzoni}
\address{(Luigi Vezzoni) Dipartimento di Matematica G. Peano \\ Universit\`a di Torino\\
Via Carlo Alberto 10\\
10123 Torino\\ Italy}
 \email{luigi.vezzoni@unito.it}
 \address{(Giulio Ciraolo) Dipartimento di Matematica e Informatica \\ Universit\`a di Palermo\\ Via Archirafi 34\\ 90123 Palermo\\ Italy}
\email{giulio.ciraolo@unipa.it}

\thanks{This work was partially supported by the project PRIN \lq\lq {\em Variet\`a reali e complesse: geometria, topologia e analisi armonica}\rq\rq ,  the projects FIRB \lq\lq {\em Differential Geometry and Geometric functions theory}\rq\rq and \lq\lq {\em Geometrical and Qualitative aspects of PDE}\rq\rq, and GNSAGA and GNAMPA (INdAM) of Italy.\\
}

\keywords{Overdetermined PDE, Rotationally symmetric spaces, Rigidity.}
    \subjclass{Primary 35R01, 35N25; Secondary: 53C24, 58J05.}

\begin{document}
\begin{abstract}
We consider a class of overdetermined problems in rotationally symmetric spaces, which reduce to the classical Serrin's overdetermined problem in the case of the Euclidean space. We prove some general integral identities for rotationally symmetric spaces which imply a rigidity result in the case of the round sphere.  
\end{abstract}

\maketitle

\section{Introduction and statement of the result}
We consider a class of overdetermined problems in rotationally symmetric spaces, which consists of an elliptic boundary value problem where Dirichlet and Neumann conditions are simultaneously imposed on the boundary. These kinds of problems are usually not well-posed and the existence of a solution imposes strong restrictions on the shape of the domain where the problem is defined. 

The study of overdetermined problems in the Euclidean space starts with the seminal paper of Serrin \cite{Se}, where solutions of the following overdetermined torsion problem 
\begin{equation} \label{pb_serrin1}
\begin{cases}
\Delta u = f(u) & \textmd{in } \Omega \,, \\
u= 0 & \textmd{on } \partial \Omega \,, \\
u_\nu = c &\textmd{on } \partial \Omega \,,
\end{cases}
\end{equation}
are considered. Here, $\Omega$ is a bounded domain, $c$ is a positive constant and $f\colon \R\to \R$ is a $C^1$-function. By using the method of the moving planes, Serrin proved that if \eqref{pb_serrin1} admits a positive (or negative) solution then $\Omega$ is a ball and the solution $u$ is radially symmetric. The result extends to more general quasilinear elliptic problems, as proved in \cite{Se,Re}.


In \cite{We} Weinberger observed that the torsion problem 
\begin{equation} \label{pb_serrin_W}
\begin{cases}
\Delta u = n & \textmd{in } \Omega \,, \\
u= 0 & \textmd{on } \partial \Omega \,, \\
u_\nu = c  &\textmd{on } \partial \Omega \,,
\end{cases}
\end{equation}
can be treated by using a simplified approach, based only on basic integral inequalities and a maximum principle for the so-called $P$-function. Notice that in this case $u<0$ in $\Omega$ in view of the maximum principle.
From the geometric point of view, the case $f\equiv n$ is special since the Euclidean metric is the Hessian metric of the radial function  $\varphi(x)=\frac{|x|^2}{2}$
and equation $\Delta u = n$ is equivalent to
$$
{\rm tr}(\nabla^2u)={\rm tr}(\nabla^2\varphi)\,,
$$ 
where $\nabla^2$ is the Hessian. With this approach, Weinberger reduced the problem to solve the equation
$$
\nabla^2 u=\nabla^2\varphi\,,
$$
instead of $\Delta u=n$ simplifying  the original problem and proving the radial symmetry by using the boundary condition $u=0$ on $\partial \Omega$. Weinberger's approach was used in \cite{FK,FGK,GL} to extend the symmetry result to more general elliptic equations.


Still by using integral identities, Serrin's theorem was extended to Hessian equations in \cite{BNST}. The proof is in the same fashion as the one of Weinberger, but without using any $P$-function.
Serrin's result was extended to the round sphere $(S^n,g_{S^n})$ in \cite{KP} and \cite{Mo} by using a moving planes approach.

The aim of the present paper is to prove a Serrin's type result on the round sphere without using the moving plane method, but employing basic integral inequalities in the same spirit as \cite{BNST,We}.  A natural obstruction in using Weinberger's approach is that the round metric on the sphere is not a Hessian metric.  
However, this approach is still successful when one replaces $\Delta u=n$ with $\Delta u=\psi$, $\psi$ being a radial function from the north pole, since it turns out that the metric $\psi\,g_{S^{n}}$ is the Hessian of a function $\varphi$ on the Hemisphere $S^n_+$ for a suitable choice of $\psi$.

Indeed, by using polar coordinates, a rotationally symmetric space can be identified with
an open ball $B$ of $\R^n$ equipped by a rotationally symmetric metric $g=dr^2 + h^2(r) g_{S^{n-1}}$. For instance, $h(r)=r$ in the Euclidean case, 
$h(r)=\sin r$ in the case of the sphere and $h(r)=\sinh r$ in the case of the hyperbolic space. By using this notation, problem \eqref{pb_serrin1} can be written in terms of $h$ as 
\begin{equation}\label{pb_serrin_h}
\begin{cases}
\Delta u = n \dot h & \textmd{ in } \Omega \,, \\
u=0 & \textmd{ on } \partial \Omega  \,, \\
u_\nu=c & \textmd{ on } \partial \Omega  \,. \\
\end{cases}
\end{equation}
Our goal is to provide a rigidity result for \eqref{pb_serrin_h} which makes use of integral identities. As far as we know, this approach has been successfully employed in the Euclidean space, but it is new for overdetermined problems in more general Riemannian manifolds. 
Our main result is the following.

\begin{theorem} \label{main2}
Let $\Omega$ be a domain in the round sphere $(S^n,g_{S^n})$, $n \geq 2$, whose closure is contained in the Hemisphere $S^n_+$.  Assume that the following problem has a solution 
\begin{equation} \label{pb_serrin}
\begin{cases}
\Delta u = n \cos(d) & \textmd{in } \Omega \,, \\
u = 0 & \textmd{on } \partial \Omega \,, \\
\partial_{\nu}u = c & \textmd{on } \partial \Omega \,,
\end{cases}
\end{equation}
where $d\colon S^n_+\to \R$ is the geodesic distance from the north pole and $c$ is a nonzero constant. Then $\Omega$ is a geodesic ball about the north pole and $u$ is a radial function. 
\end{theorem}

It is worth to be remarked that equation $\Delta u = n \cos(d) $ can be alternatively written in the $2$-dimensional hemisphere as 
$$
\tilde \Delta u = 2\,,
$$
where $\tilde \Delta$ is the Laplace operator of $\tilde g=\cos(d)\, g_{S^{2}}$. Therefore we have the following direct consequence of theorem \ref{main2}.

\begin{corollary}
Let $\Omega$ be a domain in the $2$-dimesional round sphere whose closure is contained in the Hemisphere $S^2_+$.  Let 
$\tilde g=\cos(d) g_{S^{2}}$ be the conformal change of round metric in the Hemisphere obtained by using the geodesic distance $d$ from the north pole, and let $\tilde\Delta$ be the Lapalce operator of $\tilde g$. Then  the problem  
\begin{equation*}
\begin{cases}
\tilde \Delta u = 2 & \textmd{in } \Omega \,, \\
u = 0 & \textmd{on } \partial \Omega \,, \\
\partial_{\nu}u = c & \textmd{on } \partial \Omega \,,
\end{cases}
\end{equation*}
has a solution if and only if $\Omega$ is a geodesic ball about the north pole and $u$ is a radial function. 
\end{corollary}

Our approach for proving Theorem \ref{main2} consists in providing some general integral identities in rotationally symmetric spaces which in the case of the sphere imply the result. 
We notice that in the Euclidean case $\dot h(r)=1$ and problem \eqref{pb_serrin_h} is invariant under translations. On the other hand, in the Hemisphere one has $\dot h(r) = \cos r$ and it is clear that the origin is a distinguished point. However, in spite of this different behaviour, our proof works in both cases.
We mention that our proof works also in the Euclidean space, where it is a somehow simplified version of the one in \cite{BNST} (see Corollary \ref{corol_Euclide} below).

The paper is organized as follows. In section \ref{sect_notation} we introduce some notation. In section \ref{section2} we prove a Poho\v{z}aev type identity in  rotationally symmetric spaces and in section \ref{section3}, starting from the Bochner-Weitzenb\"ock formula on Riemannian manifolds, we deduce some general identities involving solutions to \eqref{pb_serrin_h} in an arbitrary rotationally symmetric space. Moreover, we use the basic inequality 
$$
|\nabla^2u| \geq \frac{1}{n} (\Delta u)^2 
$$
to obtain a chain of integral inequalities where the equality sign is achieved if and only if the solution $u$ satisfies $\nabla^2u = \dot h g_{S^n}$.

\medskip
\noindent
 {\em{Acknowledgements}}. The authors are grateful to Antonio J. Di Scala for useful conversations about this paper.

\section{Notation} \label{sect_notation}

In this short section we declare some notation we are going to adopt in the following. Even many of them are standard, we prefer to declare them for the reader's convenience. 

Given a Riemannian manifold $(M,g)$, we denote by $\nabla$ the Levi-Civita connection of $g$ and by $\Gamma_{ij}^k$ the Christoffel symbols of $\nabla$. If $L$ is any tensor on $M$, we denote by $|L|$ the pointwise norm of $L$ with respect to $g$.
Given a $C^2$-map $u\colon M\to \R$, we denote by $Du$ the gradient of $u$, i.e. the dual field of the differential of $u$ with respect to $g$, and  by $\nabla^2u=\nabla du $ the Hessian of $u$. By definition $\nabla^2$ is a symmetric $2$-tensor whose components in local coordinates take the follow expression 
$$
(D^{2}u)_{ij}=u_{ij}-\Gamma_{ij}^ku_k\,. 
$$
We denote by $\Delta$ the Laplacian operator induced by $g$. $\Delta u$ can be defined as the trace of $\nabla^2u$ with respect to $g$ and, accordingly, it takes the following local expression
$$
\Delta u=g^{ij}u_{ij}-g^{ij}\Gamma_{ij}^ku_k\,,
$$
where $g_{ij}$ are the components of the metric $g$ and $(g^{ij})$ is the inverse matrix to $(g_{ij})$. 

Every $C^2$-map $u$ on a Riemannian manifold always satisfies the so-called  Bochner-Weitzenb\"ock formula
\begin{equation}\label{BW}
\frac12 \Delta|Du|^2=|\nabla^2u|^2+g(D(\Delta u),Du)+Ric(Du,Du)\,, 
\end{equation}
where $Ric$ is the Ricci tensor of $g$.

Given a vector field $X$ on an oriented Riemannian manifold $(M,g)$, we denote by ${\rm div}X$ the divergence of $X$ with respect to $g$. 
If ${e_k}$ is a local orthonormal frame on $(M,g)$, then
$$
{\rm div}X=_{|{\rm loc}}\sum_{k=1}^ng(\nabla_{e_k}X,e_k)\,;
$$
notice that, if $u$ is a $C^1$-map on $M$, we have 
$$
{\rm div}(uX)=g(Du,X)+u{\rm div}(X) \,.
$$ 
For a $C^1$ vector field $X$, the divergence theorem is given by
$$
\int_\Omega {\rm div} X = \int_{\partial \Omega} g(X,\nu) \,,
$$ 
where $\nu$ is the outward normal to $\Omega$ and $\Omega$ is a bounded domain regular enough. Here and in the rest of the paper, all the integrations are computed with respect to the volume form of $g$, which will be omitted.  

In this paper we mainly focus on rotationally symmetric spaces. By a {\em rotationally symmetric space} we refer to an open ball $B$ of $\R^n$ centered  at the origin $O$ (including the case $B=\R^n$) equipped by a Riemannian metric $g$ which writes in polar coordinates as  
$$
g=dr^2+h^2g_{S^{n-1}} \,,
$$
where $g_{S^{n-1}}$ is the round metric of the $(n-1)$-dimensional sphere $S^{n-1}$ and $h:[0,\bar r), \to [0,+\infty)$ is a smooth map such that $h(0)=0$ and $\dot h >0$, $\bar r$ being the radius of $B$. For instance, the Euclidean space and the Hemisphere belong to this class of examples.  Notice that if $f\colon B\to \R$ is a radial function, then its Hessian takes the following expression 
$$
\nabla^2f=\ddot f\,dr^2+\dot f h\dot h\,g_{S^{n-1}}\,.
$$
In particular if we consider a radial function $H$ which is a primitive of $h$ we have 
\begin{equation}\label{H}
\nabla^2H=\dot h\,g\,.
\end{equation}

\section{A Poho\v{z}aev-type identity in rotationally symmetric spaces}\label{section2}
%
%
%
%

A crucial tool for proving Serrin's result in \cite{BNST} and \cite{We} is the Poho\v{z}aev identity, which in the Euclidean space is given by
\begin{equation} \label{Pohozaev}
\frac{n-2}{2} \int_\Om |Du|^2 - \int_\Om (x \cdot D u) \Delta u = \frac{1}{2} \int_{\pa \Om} |Du|^2 (x \cdot \nu) - \int_{\pa \Om} (x \cdot Du) \pa_{\nu} u  \,,
\end{equation}
where $u\colon\bar \Omega\to \mathbb R$ is a function of class $C^2$ (see \cite{paul}). Our symmetry result will be based on the following generalization of the Poho\v{z}aev identity to rotationally symmetric spaces.

\begin{proposition} 
Let $(B,g=dr^2+h^2g_{S^{n-1}})$ be an $n$-dimensional rotationally symmetric space and let $X=h\partial_r$. Then  
\begin{equation}\label{identity1}
{\rm div} \left(\frac{|D u|^2}{2}X-hu_r D u\right)=\frac{n-2}{2}\dot h|D u|^2-hu_r\Delta u \,,
\end{equation}
for every map $u\colon B\to \mathbb R$ of class $C^2$. 
\end{proposition}
\begin{proof}
Let $\{\tilde e_{2},\dots,\tilde e_{n}\}$ be a local orthonormal frame of $S^{n-1}$. Then
$$
e_1=\partial_r\,,\quad e_{k}=\frac{1}{h}\tilde e_k\,,\quad k=2,\dots,n \,,
$$
 is a local orthonormal frame of $B$ with respect to $g$. In order to simplify the notation we set 
$$
f=\frac{|D u|^2}{2}\,,\quad \ell=g(X,D u)=h u_r\,. 
$$
We readily have  
$$
\begin{aligned}
{\rm div} \left(fX-\ell D u\right)&=\,g(D f,X)+f{\rm div X}-g(D \ell,D u)-\ell \Delta u\\ 
&=\,hf_r+f{\rm div X}-u_rg(D h,D u)-hg(D u_r,D u)-\ell \Delta u\\ 
&=\,hf_r+f{\rm div X}-\dot h(u_r)^2-hg(D u_r,D u)-\ell \Delta u\,.
\end{aligned}
$$
Moreover,  
$$
{\rm div X}=\sum_{k=1}^ng(\nabla_{e_k}h\partial_r,e_k)=\dot h+\sum_{k=1}^n hg(\nabla_{e_k}\partial_r,e_k)\,.
$$
Furthermore the Koszul formula implies that for $k\neq 1$ we have 
$$
g(\nabla_{e_k}\partial_r,e_k)=g([e_k,\partial_r],e_k)=g([h^{-1}\tilde e_k,\partial_r], h^{-1}\tilde e_k)=\frac{\dot h}{h}
$$
and so 
\begin{equation}\label{divX}
{\rm div X}= n\dot h\,.
\end{equation}
Now,
$$
\begin{aligned}
D u_r=&\,\sum_{k=1}^n e_k(\partial_ru)e_k=u_{rr}+\sum_{k=2}^{n} h^{-1}\tilde e_k(\partial_ru)e_k=
u_{rr}\partial_r+\sum_{k=2}^{n} h^{-1} \partial_r(\tilde e_ku)e_k\\
=&\,\sum_{k=1}^{n}  \partial_r(e_ku)e_k+\sum_{k=2}^{n}  h^{-1}\dot h( e_ku)e_k
\end{aligned}
$$
and 
$$
f_r=\frac12\,\partial_rg(D u,D u)=\sum_{k=1}^n\partial_r(e_k(u))e_k(u)\,,
$$
imply that 
$$
f_r-g(D u_r,D u)=-\sum_{k=2}^{n}  h^{-1}\dot h( e_ku)e_k(u)\,.
$$
So 
$$
{\rm div} \left(fX-\ell D u\right)
=nf\dot h-\dot h|D u|^2-\ell \Delta u=\frac{n-2}{2}\dot h|D u|^2-h\partial_r u\Delta u\,,
$$
i.e. 
$$
{\rm div} \left(\frac{|D u|_g^2}{2}X-hu_r D u\right)=\frac{n-2}{2}\dot h|D u|^2-hu_r\Delta u\,,
$$
as required. 
\end{proof}

\section{General results on rotational symmetric spaces}\label{section3}
In this section we focus on problem \eqref{pb_serrin_h} on rotationally symmetric spaces. The main result of this section is the following proposition. 
\begin{proposition}\label{pre1}
Let $(B,g=dr^2+h^2\,g_{S^{n-1}})$ be a rotationally symmetric space and let $\Omega$ be a bounded domain whose closure is contained in $B$.  Then every $C^2$-solution $u$ to 
\begin{equation} \label{pb_serrin2}
\begin{cases}
\Delta u = n \dot h & \textmd{in } \Omega \,, \\
u = 0 & \textmd{on } \partial \Omega \,, \\
\partial_{\nu}u = c & \textmd{on } \partial \Omega \,,
\end{cases}
\end{equation}
where $c$ is a positive constant, satisfies 
\begin{equation}\label{BW2}
 \frac12\int_\Omega(-u)\Delta|Du|^2\geq n\int_\Omega-u\dot h^2+n\int_\Omega-uu_r\ddot h+\int_\Omega -uRic(Du,Du)\,,
\end{equation}
and the equality sign holds if and only if 
$$
\nabla^2 u=\nabla^2(H)\,,
$$
where $H$ is a primitive of $h$.
Furthermore 
\begin{equation}
\label{left2}
 \frac12\int_\Omega(-u)\Delta|Du|^2=n\int_{\Omega}-u\dot h^2+n\int_{\Omega}-uh\ddot h-(n-1)\int_{\Omega}-uu_r\ddot h \,.
\end{equation}
\end{proposition}

%

\begin{proof}
Inequality \eqref{BW2} can be easily obtained by combining the Bochner-Weitzenb\"ock formula and the Newton inequality. Indeded,  by applying \eqref{BW} to a solution of \eqref{pb_serrin2} and taking into account the Newton inequality 
\begin{equation}\label{newton}
n|\nabla^2u|^2\geq (\Delta u)^2\,,
\end{equation} 
we get 
\begin{equation}\label{BW}
\frac12\Delta|Du|^2_g\geq\frac1n(\Delta u)^2+ n u_r\ddot h+ Ric(Du,Du)
\end{equation}
pointwise. Since $\dot h >0$, then $u\leq 0$ in $\bar \Omega$ and from $\Delta u=n\dot h$ we obtain \eqref{BW2}.

Furthermore, if the equality sign in \eqref{BW2} holds, then 
$$
\int_\Omega -u \left(|\nabla^2 u|^2 - \frac{1}{n}(\Delta u )^2\right) = 0 \,,
$$
which implies that the equality sign 
holds in Newton's inequality \eqref{newton} and, since $\Delta u=n\dot h$, we obtain 
$$
\nabla^2 u=\dot h\,g\,.
$$
Moreover, if $H$ is a primitive of $h$, then formula \eqref{H} implies 
$$
\nabla^2u=\nabla^2H\,,
$$ 
as required. 

The proof of \eqref{left2} is more involved. 
Since 
$$
{\rm div }(-u\,D|Du|^2)=-g(Du,D|Du|^2)-u\Delta |Du|^2
$$
and $u$ vanishes on $\partial \Omega$,  the divergence theorem implies 
$$
\int_\Omega(-u)\Delta|Du|^2=\int_{\Omega} g(Du,D|Du|^2)\,.
$$
Moreover,  
$$
{\rm div }(|Du|^2Du)=g(Du,D|Du|^2)+|Du|^2\Delta u=g(Du,D|Du|^2)+n\dot h|Du|^2
$$
and so, keeping in mind that $Du=c\nu$, we have 
$$
\begin{aligned}
 \frac12\int_\Omega(-u)\Delta|Du|^2=&\, \frac12\int_{\Omega} {\rm div }(|Du|^2Du)- \frac n2\int_\Omega \dot h|Du|^2\\
=&\, \frac12\int_{\partial \Omega} |Du|^2g(Du,\nu)- \frac n2\int_\Omega \dot h|Du|^2\\
=&\, \frac{c^3}2|\partial \Omega|- \frac n2\int_\Omega \dot h|Du|^2\,,
\end{aligned}
$$
i.e., 
\begin{equation}
\label{left}
 \frac12\int_\Omega(-u)\Delta|Du|^2=\frac{c^3}2|\partial \Omega|- \frac n2\int_\Omega \dot h|Du|^2\,.
\end{equation}
Since $u=0$ on $\partial \Omega$,  we have
$$
\int_\Omega \dot h |Du|^2=\int_{\Omega}{\rm div} (\dot huDu)-\int_{\Omega}\dot hu\Delta u-\int_{\Omega}ug(D\dot h,Du)=n\int_{\Omega}-u\dot h^2+\int_{\Omega} -uu_r\ddot h\,,
$$
i.e.
\begin{equation}\label{20}
\int_\Omega \dot h |Du|^2=n\int_{\Omega}-u\dot h^2+\int_{\Omega} -uu_r\ddot h\,.
\end{equation}

In order to treat $|\partial\Omega|$ in \eqref{left} and in \eqref{BW2} we need to show some preliminary formulas.  
We first notice that since $u$ satisfies \eqref{pb_serrin2} then
$$
|\partial \Omega|=\int_{\partial \Omega}\frac{1}{c}\partial_{\nu}u=\frac{1}{c}\int_{\Omega}\Delta u=\frac{n}{c}\int_{\Omega}\dot h\,,
$$
which implies 
\begin{equation} \label{17} 
|\partial \Omega|=\frac{n}{c}\int_{\Omega}\dot h \,.
\end{equation}

Furthermore we need to show that 
\begin{equation}\label{18}
c^2\int_{\Omega}\dot h=\int_{\Omega}-u\left((n+2)\dot h^2+2h\ddot h\right)-\frac{n-2}{n}\int_{\Omega}-uu_r\ddot h\,.
\end{equation}
To prove \eqref{18}, we integrate \eqref{identity1} and obtain  
$$
\int_\Omega{\rm div} \left(\frac{|D u|^2}{2}X-hu_r D u\right)=\frac{n-2}{2}\int_{\Omega}\dot h|D u|^2-\int_{\Omega}hu_r\Delta u \,.
$$
Since $\nu=\frac{1}{c}Du$ on  $\partial \Omega$, we have
$$
\begin{aligned}
\int_\Omega{\rm div} \left(\frac{|D u|^2}{2}X-hu_r D u\right)=&\,\frac{c^2}{2} \int_{\partial\Omega}g(X,\nu)-c\int_{\partial \Omega} hu_r=
\frac{c^2}{2} \int_{\partial\Omega}g(X,\nu)-c^2\int_{\partial\Omega}g(X,\nu)\\
=&\,-\frac{c^2}{2}  \int_{\partial\Omega}g(X,\nu)=-\frac{c^2}{2}\int_{\Omega}{\rm div} X=-\frac{nc^2}{2}\int_{\Omega}\dot h \,,
\end{aligned}
$$
and equation \eqref{20} implies 
\begin{equation}\label{x}
-\frac{c^2n}{2}\int_{\Omega}\dot h=\frac{n-2}{2}\int_{\Omega}\left(-un\dot h^2-uu_r\ddot h\right)-n\int_{\Omega}
u_rh\dot h\,.
\end{equation}
Since 
\begin{eqnarray*}
&& u_rh\dot h=g(Du,D(h^2/2))\,;\\
&& {\rm div}(uD( h^2/2))=g(Du,D( h^2/2))+u\Delta (h^2/2)\,;\\
&& \Delta (h^2/2)=n\dot h^2+h\ddot h\,,
\end{eqnarray*}
we have, 
$$
\int_\Omega u_rh\dot h=\int_{\Omega}{\rm div}\left( uD(h^2/2)\right)-\int_{\Omega}u\Delta (h^2/2)=\int_{\Omega} -u\left(
n\dot h^2+h\ddot h\right) \,.
$$
Hence equation \eqref{x} implies 
$$
-\frac{c^2n}{2}\int_{\Omega}\dot h=\frac{n-2}{2}\int_{\Omega}\left(-un\dot h^2-uu_r\ddot h\right)-n\int_{\Omega} -u\left(
n\dot h^2+h\ddot h\right)\,
$$
which is equivalent to \eqref{18}. 

Combining \eqref{17} and \eqref{18}, we get 
$$
\frac{c^3}2|\partial \Omega|=\frac{c^2n}{2}\int_{\Omega}\dot h=
\frac n2\int_{\Omega}-u\left((n+2)\dot h^2+2h\ddot h\right)-\frac{n-2}{2}\int_{\Omega}-uu_r\ddot h
$$
and then, by taking into account \eqref{left} and \eqref{20}, we get \eqref{left2}.
\end{proof}

Proposition \ref{pre1} readily implies Serrin's theorem for problem \eqref{pb_serrin1} in the Euclidean space. Namely,

\begin{corollary}\label{corol_Euclide}
Let $\Omega\subseteq \R^n$ be a bounded domain in the Euclidean space and assume that problem \ref{pb_serrin1} has a solution in $\Omega$. Then $\Omega$ is a ball and $u$ is radial. 
\end{corollary}
\begin{proof}
In the Euclidean case, i.e. when $h(r)=r$ and $B=\R^n$, equations \eqref{BW2} and \eqref{left2} agree. So in particular the  
equality sign in \eqref{BW2} holds and $\nabla^2 u= Id \, $ which, together with the condition $u=0$ on $\partial \Omega$, implies that $u=|x-x_0|^2 - R^2$ and $\Omega=B_R(x_0)$, for some $x_0\in \mathbb{R}^n$ and $R>0$.
\end{proof}

In contrast to the Euclidean case, theorem \ref{main2} does not directly follow from proposition \ref{pre1}, and we need to prove the next proposition. 


\begin{proposition}\label{lemma_pino}
Under the same assumptions of  proposition \ref{pre1}, we have
\begin{eqnarray}
&&\label{quelo} \int_{\Omega} -u|Du|^2=\int_{\Omega} -uhu_r \,, \\
&&\label{seconda} \int_{\Omega}-u\left(|Du|^2-h^2\right)\leq 0\,. 
\end{eqnarray}
\end{proposition}

\begin{proof}
Let $u$ be a $C^2$-solution to \eqref{pb_serrin2}, then \eqref{divX} implies  
$$
{\rm div}\left(\frac{u^2}{2} h\,\partial_r\right)=uu_rh+n\frac{u^2}{2}\dot h
$$
and, by taking into account that $u$ vanishes on $\partial \Omega$, we get  
$$
\int_{\Omega} \frac{u^2}{2}n\dot h=\int_{\Omega} -uhu_r \,;
$$
hence
\begin{equation*}
\int_{\Omega} -u|Du|^2=\int_{\Omega} \frac{u^2}{2}\Delta u=\int_{\Omega} \frac{u^2}{2}n\dot h=\int_{\Omega} -uhu_r\,,
\end{equation*}
which gives \eqref{quelo}. Moreover,
\begin{equation*}
\int_{\Omega} -u|Du|^2=\int_{\Omega} -uhu_r\,\leq \frac12 \int_\Omega -uh^2+\frac12 \int_\Omega -uu_r^2\leq 
 \frac12 \int_\Omega -uh^2+\frac12 \int_\Omega -u|Du|^2\,,
\end{equation*}
where in the last step we have used $u_r^2\leq |Du|^2$,
which implies 
$$
\int_{\Omega}-u\left(|Du|^2-h^2\right)\leq 0\,,
$$  
as required. 
\end{proof}

\section{Proof of theorem \ref{main2}}
In this section we prove theorem $\ref{main2}$ and give some remark.

\begin{proof}[Proof of theorem $\ref{main2}$]
The hemisphere $(S^{n}_+,g_{S^{n}})$ can be regarded as the open ball $B$ of $\R^n$ of radius $\pi/2$ equipped with the rotationally symmetric metric 
$$
g=dr^2+h^2\,g_{S^{n-1}}\,,
$$ 
where $h(r)=\sin (r)$. Problem \eqref{pb_serrin} writes in this setting as \eqref{pb_serrin2}. 
 
Since 
$$
\ddot h=-h\,,\quad Ric(Du,Du)=(n-1)|Du|^2\,,
$$
by using \eqref{left2} and \eqref{quelo}, we get that in this case \eqref{BW2} is equivalent to 
$$
\int_\Omega -u|Du|^2 - \int_{\Omega}-uh^2 \geq 0 \,,
$$
which is the reverse inequality of \eqref{seconda}.  Therefore the equality sign in \eqref{BW2} holds and proposition \ref{pre1} implies that   
$$
D^{2}(u-H)=0\,,
$$
where $H$ is a primitive of $h$. 
It follows that $u-H$ is constant and so $u$ is a primitive of $h$. In particular $u$ is a radial function and $\Omega$ is a ball centered at $O$.   
\end{proof}

\begin{remark}{\rm 
It is quite natural to ask if an analogue of theorem \ref{main2} can be proven on the Hyperbolic space. Unfortunately, our argument fails in the hyperbolic space since \eqref{BW2} reduces to \eqref{seconda}. }
\end{remark}

\end{document}